\documentclass[12pt,a4paper,reqno]{amsart}

\usepackage{geometry}
\usepackage{hyperref}
\usepackage{graphicx}
\usepackage{subcaption}
\usepackage{amsmath,amsthm,amssymb}
\usepackage{xcolor}

\newtheorem{theorem}{Theorem}

\newtheorem{lemma}{Lemma}
\theoremstyle{remark}
\newtheorem*{remark}{Remark}

\numberwithin{equation}{section}

\renewcommand{\Re}{\mathfrak{Re}}

\begin{document}

\title[Zeta at its local extrema]{The second moment of the Riemann zeta function at its local extrema}
\author[C. Hughes]{Christopher Hughes}
\address{Department of Mathematics, University of York, York, YO10 5DD, United Kingdom}
\email{christopher.hughes@york.ac.uk}
\author[S. Lugmayer]{Solomon Lugmayer}
\address{Department of Mathematics, University of York, York, YO10 5DD, United Kingdom}
\email{solomon.lugmayer@york.ac.uk}
\author[A. Pearce-Crump]{Andrew Pearce-Crump}
\address{School of Mathematics, University of Bristol, Bristol, BS8 1UG, United Kingdom}
\email{andrew.pearce-crump@bristol.ac.uk}

\begin{abstract}
Conrey and Ghosh studied the second moment of the Riemann zeta function, evaluated at its local extrema along the critical line, finding the leading order behaviour to be
$\frac{e^2 - 5}{2 \pi} T (\log T)^2$. This problem is closely related to a mixed moment of the Riemann zeta function and its derivative.

We present a new approach which will uncover the lower order terms for the second moment as a descending chain of powers of logarithms in the asymptotic expansion.
\end{abstract}

\maketitle

\section{Background}
Under the Riemann Hypothesis, Conrey and Ghosh \cite{ConGho85} showed that
    \begin{equation}\label{CGLead}
       \sum_{0< \gamma \leq T} \max_{\gamma \leq t \leq \gamma^+} \left| \zeta \left( \frac{1}{2} + it \right) \right|^2 \sim \frac{e^2 - 5}{4 \pi} T (\log T)^2
    \end{equation}
as $T\rightarrow \infty$ where $\gamma \leq \gamma^+$ are successive ordinates of non-trivial zeros of $\zeta (s)$. In this paper, we will show how to obtain the lower order terms of this asymptotic, with the result given in Theorem \ref{thm:LowerOrderTerms}.

The Hardy $Z$-function $Z(t)$ is the real function of a real variable which satisfies $$|Z(t)| = |\zeta (1/2 +it)|,$$ so Conrey and Ghosh's result can be re-written in terms of the Hardy Z-function. It will be more convenient to work with the Hardy Z-function in various places throughout this paper, but we have endeavoured to link this back to the Riemann zeta function where applicable.

Since we are assuming the Riemann Hypothesis throughout this paper, we have that the zeros of $Z'(t)$ are interlaced with those of $\zeta(\frac{1}{2}+it)$, and hence with those of $Z(t)$.

Conrey and Ghosh further showed in \cite{ConGho89} that this sum over the maxima can be rewritten in terms of a joint moment of the Z-function and its derivative, showing
    \begin{equation}\label{CGIntegral}
        \int_{1}^{T} |Z(t)||Z'(t)| \ dt \sim \frac{e^2 - 5}{4 \pi} T (\log T)^2.
    \end{equation}
We note that the corresponding value with the Hardy Z-function replaced by the Riemann zeta function is still unknown.

Conrey \cite{Con88} established under the Riemann Hypothesis that
\begin{equation*}
   \frac{\sqrt{21}}{90 \pi^2} T (\log T)^5 \lesssim \sum_{0< \gamma \leq T} \max_{\gamma \leq t \leq \gamma^+} \left| \zeta \left( \frac{1}{2} + it \right) \right|^4 \lesssim \frac{1}{2 \sqrt{15} \pi^2} T (\log T)^5,
\end{equation*}
where $A \lesssim B$ means $A \leq B(1+o(1))$, with some improvements on the upper bound due to Hall \cite{Hall03,Hall06}.

Although no other asymptotics for higher moments are known, Milinovich \cite{Milinovich11} 
established upper and lower bounds. He showed that for $k \in \mathbb{N}$ and for $\varepsilon >0$,
    \begin{equation*}
       T(\log T)^{k^2 +1 - \varepsilon} \ll  \sum_{0< \gamma \leq T} \max_{\gamma < t \leq \gamma^+} |Z(t)|^{2k} \ll T(\log T)^{k^2 +1 + \varepsilon}
    \end{equation*}
for sufficiently large $T$. He also gives an argument showing
    \begin{equation*}
        \sum_{0< \gamma \leq T} \max_{\gamma < t \leq \gamma^+} |Z(t)|^{2k} = k \int_{1}^{T} |Z(t)|^{2k-1} |Z'(t)| \ dt + O_{k,\varepsilon}(T^\varepsilon).
    \end{equation*}
Finally, he conjectures for $k \in \mathbb{N}$ that
    \begin{equation*}
        \sum_{0< \gamma \leq T} \max_{\gamma < t \leq \gamma^+} |Z(t)|^{2k} \sim C_k T (\log T)^{k^2 +1}
    \end{equation*}
as $T \rightarrow \infty$, but did not conjecture the value of the implied constant.

The groundbreaking work of Keating and Snaith \cite{KS00a,KS00} changed all of this, when they used characteristic polynomials of random unitary matrices to conjecture moments of the Riemann zeta function of the form.

Following the groundbreaking work of Keating and Snaith \cite{KS00a,KS00}, the problem of calculating joint moments of the Riemann zeta function and its derivative was taken up by Hughes \cite{Hughes01} in his thesis. Both the approaches involved calculating analogous moments of characteristic polynomials of random unitary matrices. Hughes conjectured the coefficient $F(s,h)$ for integer $s,h$ in terms of combinatorial sums, where $F(s,h)$ is given in
\begin{equation*}
    \int_0^T |Z(t)|^{2s-2h} |Z'(t)|^{2h} \ dt \sim F(s,h) T (\log T)^{s^2 +2h}
\end{equation*}
as $T \rightarrow \infty$. The conjecture agreed with all values known at the time, namely $F(1,1)$ by Ingham \cite{Ing26}, and $F(2,1)$ and $F(2,2)$ by Conrey \cite{Con88} and independently Hall \cite{Hall99}. This expression for $F(s, h)$ made sense for real $s$ as
well, but not general real $h$. Note that if we could set $h=1/2$, then we could infer Milinovich's conjecture with the constant given explicitly.

Alternative approaches yielding different representations of the conjectured $F(s,h)$ with $s,h \in \mathbb{N}$ were given by Conrey, Rubinstein and Snaith \cite{ConRubSna06} (in the special case where $s=h$), by Dehaye \cite{Deh08,Deh10}, by Basor et al.~\cite{Basor19}, and by Bailey et al.~\cite{Bailey19}. Links between $F(s,h)$ and Painlev\'{e} equations have been noted by Forrester and Witte \cite{ForWit06}, and Basor et al.~\cite{Basor19}.

However, the problem of non-integer $s$ and $h$ is much more difficult. Winn \cite{Winn} was able to establish results for $s \in \mathbb{N}$ and $h \in \frac{1}{2} \mathbb{N}$ by finding connections with hypergeometric functions, giving an expression for $F(s,h)$ in terms of a combinatorial sum. In particular, for $s=1$ and $h=1/2$ he established the random matrix theory (RMT) version of \eqref{CGIntegral} obtaining $(e^2 - 5) / 4 \pi$.

Assiotis, Keating and Warren \cite{AKW22} proved results on the random matrix side for arbitrary real values of $s > - 1/2$ and positive real values of $h$ in the full range $0 < h < s + 1/2$.

In addition to the work on the typical joint moment, Keating and Wei \cite{KeatWei231,KeatWei232} have considered the RMT analogue of joint moments for higher derivatives. The case of the joint moments of arbitrary numbers of higher
order derivatives with arbitrary positive real exponents has been settled by Assiotis, Gunes, Keating and Wei \cite{AGKW24}.

Curran \cite{Cur21} unconditionally obtained upper bounds of the conjectured order for the joint moments in the range $1 \leq s \leq 2$ and $0 \leq h \leq 1$. Curran and Heycock \cite{CurHey24} extend these upper bounds to all $0 \leq h \leq s \leq 2$, and obtain lower bounds for all $0 \leq h \leq s+1/2$. By assuming the Riemann Hypothesis, they give sharp bounds for all $0 \leq h \leq s$. They also prove upper bounds of the conjectured order for joint moments of the Riemann zeta function with its higher derivatives.

\section{Results}
As mentioned previously, the purpose of this paper is to give a full asymptotic with lower-order terms for \eqref{CGLead}. Specifically, we prove the following theorem.

\begin{theorem}\label{thm:LowerOrderTerms}
Assume the Riemann Hypothesis. Given $\gamma$, an ordinate of a non-trivial zero of $\zeta (s)$, let $\gamma^+\geq \gamma$ denote the next successive ordinate of a zeta zero. Let
$L = \log \frac{T}{2\pi}$.
Then for any fixed $N\geq 0$, as $T\to\infty$ we have
\[
   \sum_{0< \gamma \leq T} \max_{\gamma \leq t \leq \gamma^+} \left| \zeta \left( \frac{1}{2} + it \right) \right|^2   =  \frac{e^2-5}{2} \frac{T}{2\pi} L^2 +  \alpha_{-1}\frac{T}{2\pi} L + \frac{T}{2 \pi} \sum_{n=0}^{N} \frac{\alpha_n}{L^n} + O_N\left(\frac{T}{L^{N+1}}\right)
\]
where
\begin{gather*}
\alpha_{-1} = 5-e^2-10 \gamma_0 +2 e^2 \gamma_0  \\
\alpha_0 = 12 \gamma_1-4 e^2 \gamma_1-5+e^2+10 \gamma_0 -2 e^2 \gamma_0 -4\gamma_0 ^2
\end{gather*}
and for $n \geq 1$
\[
\alpha_n = \sum_{k=n}^\infty \frac{2^{k+1}  c_{k,n+2} }{(k-n)!} +  (n-1)! \sum_{k=1}^\infty \frac{2^{k+1}}{(k-1)!} \sum_{j=0}^{\min\{k,n-1\}}  \binom{k}{j} c_{k,j+2}
\]
where the $c_{k,\ell}$ are the Laurent series coefficients around $s=1$ of\[
\left(\frac{\zeta '}{\zeta} (s) \right)' \left(-\frac{\zeta '}{\zeta} (s) \right)^{k-1} \zeta^2 (s) \frac{1}{s} = \sum_{\ell=0}^\infty c_{k,\ell} (s-1)^{-k-3+\ell} .
\]
\end{theorem}

\begin{remark} \hfill
\begin{enumerate}
    \item This recovers the leading-order asymptotic behaviour of Conrey and Ghosh.
    \item This infinite descending chain of powers of $\log T$ is due to a pole in the function that we consider near $1$, given by $\beta_1 =1+2 / L+O\left(L^{-2}\right)$. Were we able to find a closed form for the asymptotic as a function of $\beta_1$, we expect we would be able to give Theorem \ref{thm:LowerOrderTerms} with a power-saving error term of $O \left(T^{1/2+\varepsilon} \right)$ for all $\varepsilon>0$. We have set aside the search for such a form for the time being.
\end{enumerate}
\end{remark}

\section*{Acknowledgements}
We would like to thank Brian Conrey for encouraging us to pursue this project. We would also like to acknowledge the University of York, for both supporting the first author during his research leave where the idea for this paper was first investigated, and for funding the third author's PhD, during which this work was carried out.

\section{Overview of the paper}
In Section \ref{Lemmas}, we consider properties our auxiliary function
\[
Z_1(s)=\zeta^{\prime}(s)-\frac{1}{2} \frac{\chi^{\prime}}{\chi}(s) \zeta(s)
\]
used in the proof, where $\chi(s)$ is the factor from the non-symmetrical form of the functional equation of the Riemann zeta function. This function is zero at the extrema of the Riemann zeta function.

We state and prove some preliminary lemmas required for the wider proof, including a twisted second moment of the Riemann zeta function, and some identities of arithmetic functions.

In Section \ref{Proof}, we prove the theorem. This involves writing the sum over the extrema as the integral
\begin{equation*}
\sum_{0< \gamma \leq T} \max_{\gamma \leq t \leq \gamma^+} \left| \zeta \left( \frac{1}{2} + it \right) \right|^2  = \frac{1}{2\pi i} \int_{\mathcal{C}} \frac{Z_1'}{Z_1}(s) \zeta(s) \zeta(1-s) \ ds
\end{equation*}
(up to a small error), where $\mathcal{C}$ is a contour we specify later. We note at this point one of the main differences in our method with that of Conrey and Ghosh. Their method forces them to consider their integral over a short height (from $T$ to $T+T^{3/4}$), whereas we are able to define our contour across a full height of length $T$, avoiding any complications arising from summing over dyadic intervals.

We begin by showing various parts of the respective contour integrals are small, and applying the twisted second moment result mentioned previously. The bulk of the argument then comes down to evaluating
\[
\frac{1}{2 \pi i} \int_{c+i}^{c+iT}  \chi(1-s) \frac{Z_1^{\prime}}{Z_1}(s) \zeta(s)^2 \ ds
\]
for some $c>1$ that we define later.

We apply various arithmetic identities to the logarithmic derivative of $Z_1 (s)$ in the integrand, giving a series of the form
\[
\frac{Z_1'}{Z_1}(s) = \sum_{n=1}^\infty \frac{a(n,s)}{n^s}
\]
which converges for some $c>1$ with coefficients $a(n,s)$ that we specify later. Using this and the Dirichlet series for the square of the zeta function allows us to apply the method of stationary phase. That is, we can rewrite the integral as
\begin{equation*}
- \sum_{nm \leq T/2\pi} \Lambda (n) d(m) + \sum_{k=1}^{\infty} 2^k \sum_{nm \leq T/2 \pi}  \frac{a_k(n)}{(\log nm)^k} d(m)
\end{equation*}
up to some small errors, where the $a_k(n)$ come from the coefficients $a(n,s)$. We then apply Perron's formula to turn the sums back into integrals, where the first summation is standard and the second summation over $k$ is related to
\begin{equation*}
    \frac{1}{2 \pi i} \int_{(c)} \left(\frac{\zeta '}{\zeta} (s) \right)' \left(-\frac{\zeta '}{\zeta} (s) \right)^{k-1} \zeta^2 (s) \frac{x^s}{s} \ ds,
\end{equation*}
and we complete the argument with partial summation. It is at this point that we are able to highlight the second main difference between our approach and that of Conrey and Ghosh, which leads to the correct lower-order term expansion, and explains why their approach must fail to give the correct expansion.

In Section \ref{Graphs}, we look at some numerical data related to our result, and describe how we were able to perform these numerical calculations.  Finally, we plot some data to show how perfectly our results fit the truth, as well as how many terms in the asymptotic are required until it is clear that we are mostly just looking at error terms (even though there are infinitely many more terms left in the asymptotic expansion).

\section{Preliminary Lemmas}\label{Lemmas}

To motivate the initial steps of the proof, assuming the Riemann Hypothesis, the points that we wish to sum over are the real zeros of the derivative of Hardy's function. In the complex plane, we can write Hardy's function as
\[
Z(s) = \chi(1-s)^{1/2} \zeta(s)
\]
where $\chi(s)$ is the term in the functional equation for zeta, which satisfies $\chi(s) \chi(1-s)=1$. Logarithmically differentiating we get
\begin{align*}
\frac{Z'}{Z}(s) &= -\frac12 \frac{\chi'}{\chi}(1-s) + \frac{\zeta'}{\zeta}(s)\\
&= -\frac12 \frac{\chi'}{\chi}(s) + \frac{\zeta'}{\zeta}(s)
\end{align*}

This suggests an appropriate function to consider is
\[
Z_1(s)=\zeta^{\prime}(s)-\frac{1}{2} \frac{\chi^{\prime}}{\chi}(s) \zeta(s).
\]
We see that $Z_1(s)$ is zero where $Z ' (t) = 0$, that is, at an extrema of $|\zeta(s)|$ on the critical line.

This function satisfies various properties, all of which are proved by Conrey and Ghosh in the Lemma in \cite{ConGho85}. We refer the reader to their paper for the proof, and list only the properties that we need here.

\begin{lemma}\label{lem:CGLemma}
    We have the following properties for $Z_1(s)$:
    \begin{enumerate}
        \item $|Z_1(1/2 + it)| = |Z'(t)|$.
        \item $Z_1(s)$ satisfies the functional equation
        \[
        Z_1(s) = -\chi(s) Z_1 (1-s)
        \]
        for all $s$, where $\chi (s)$ is the term in the functional equation for the Riemann zeta function, $\zeta (s) = \chi (s) \zeta (1-s)$.
        \item The number of zeros of $Z_1(t)$ up to a height $T$ which lie off the critical line is bounded by $\log T$.
        \item If $\ Z_1(\beta_1 + i \gamma_1) = 0$ then
        \[
        \left|\beta_1 - \frac{1}{2} \right| \leq \frac{1}{9}
        \]
        for $\gamma_1$ sufficiently large.
    \end{enumerate}
\end{lemma}

We will also need to calculate a twisted second moment of the Riemann zeta function.
\begin{lemma}\label{lem:TwistSecond}
    \begin{multline}\label{eq:TwistSecond}
    - \frac{1}{2 \pi} \int_{1}^{T}\frac{\chi^{\prime}}{\chi}\left(\frac{1}{2}+it\right)\left|\zeta \left(\frac{1}{2}+it \right)\right|^2 \ dt = \\ \frac{T}{2\pi} \left( \log \frac{T}{2\pi} \right)^2 + \frac{T}{2\pi} \log \frac{T}{2\pi} (-2 +2\gamma_0) + \frac{T}{2\pi} (2-2\gamma_0) + O\left(T^{1/2+\varepsilon}\right),
\end{multline}
where $\gamma_0$ is Euler's constant.
\end{lemma}

\begin{proof}
We begin with the estimate, which can be found in \cite{ConGho85}, for example,
    \begin{equation}\label{eq:ChiPrimeOverChi}
    \frac{\chi'}{\chi} (s) = - \log \frac{|t|}{2 \pi} + O \left(\frac{1}{1 + |t|} \right),
    \end{equation}
where $s=\sigma + it$, with $|\sigma| \leq 2$.

Recall also the second moment of the Riemann zeta function, due to Ingham \cite{Ing26}, that
    \[
    \int_{1}^{T} \left|\zeta \left(\frac{1}{2}+it \right) \right|^2 \ dt = T \log \frac{T}{2\pi} + (2 \gamma_0-1) T + O\left(T^{1/2+\varepsilon}\right).
    \]

Integration by parts then gives the result of the lemma.
\end{proof}

\begin{lemma}\label{Mobius}
    We have the following two M\"obius identities:
\[
\sum_{d \mid n} \mu(d) \log(\tfrac{n}{d}) = \Lambda(n)
\]
and Selberg's identity
\[
\sum_{d \mid n} \mu(d) \log^2(\tfrac{n}{d}) = \Lambda(n) \log(n)  + (\Lambda \ast \Lambda)(n)
\]
where $\mu (n)$ is the M\"obius function and $\Lambda (n)$ is the von Mangoldt function.
\end{lemma}

We will also need the logarithmic derivative for $Z_1(s)$, which we note here for completeness as
\begin{equation}\label{eq:logderivZ1FE}
    \frac{Z_1'}{Z_1}(s) = \frac{\chi '}{\chi} (s) - \frac{Z_1'}{Z_1}(1-s).
\end{equation}
This follows from (2) in Lemma \ref{lem:CGLemma}.

Finally, we need a Dirichlet series for this logarithmic derivative of $Z_1 (s)$, or at least an expansion that looks like a Dirichlet series.
\begin{lemma}\label{lem:Dirichletlogderiv}
    For $\Re(s)> 1$ and $t \geq 100$, we have
\[
\frac{Z_1'}{Z_1}(s) = \sum_{n=1}^\infty \frac{a(n,s)}{n^s} + O \left( \frac{1}{t \log t} \right)
\]
where
\begin{equation*}
    a(n,s) = -\Lambda (n) + \sum_{k=1}^{\infty} \frac{1}{f(s)^k} a_k(n)
\end{equation*}
with
\begin{equation}\label{eq:f(s)}
f(s) = - \frac{1}{2} \frac{\chi '}{\chi} (s)
\end{equation}
and $a_k(n)$ is of the form
\begin{equation}\label{eq:ak(n)}
a_k(n) = \left((\Lambda \log) \ast \Lambda_{k-1}\right)(n),
\end{equation}
where we use the notation $\Lambda_k$ to denote
\[
\underbrace{\Lambda \ast \Lambda \ast \Lambda \ast \dots  \ast \Lambda}_{k-1 \text{ convolutions}}
\]
with the convention that $\Lambda_0(n)$ takes the value $1$ if $n=1$ and $0$ otherwise, and where $\Lambda_1 (n)= \Lambda (n)$ is the usual von Mangoldt function, and for clarity $\Lambda_2 (n) = (\Lambda \ast \Lambda)(n)$.
\end{lemma}

\begin{remark}
This way of writing the logarithmic derivative of $Z_1(s)$ is another point of deviation from Conrey and Ghosh's method. They set $f(s) = \frac{1}{2} \log T$ and treat this as a constant. By keeping $f(s)$ exact, we are able to retain more information for the asymptotic, as we will see later in the argument.
\end{remark}

\begin{proof}
We begin by writing
\[
Z_1(s) = \zeta'(s) + f(s) \zeta(s)
\]
where $f(s) = - \frac12 \frac{\chi'}{\chi}(s)$. Taking logarithmic derivatives we have
\begin{equation}\label{eq:logderivZ1}
    \frac{Z_1'}{Z_1}(s) = \frac{\zeta''(s) +  f(s) \zeta'(s)}{\zeta'(s) + f(s) \zeta(s)} + f'(s) \frac{1}{\frac{\zeta'}{\zeta}(s) + f(s)}
\end{equation}
For $\Re(s)>1$ with $t\geq 1$, the last term can be bounded by $O(1/t \log t)$ by noting that $f'(s) = O(1/t)$ (which follows from \eqref{eq:ChiPrimeOverChi}), $f(s) = O(\log t)$ and $\zeta'/\zeta(s) = O(\log\log t)$ (due to Littlewood \cite{Lit24}, under the Riemann Hypothesis) in the region under consideration.

For $s$ with sufficiently large real part, it is sensible to factor the denominator as
\begin{align*}
    \frac{1}{\zeta'(s) + f(s) \zeta(s)} &= \frac{1}{f(s)} \frac{1}{\zeta(s)}\left(1 + \frac{1}{f(s)}  \frac{\zeta'}{\zeta}(s) \right)^{-1}\\
    &= \frac{1}{f(s)} \frac{1}{\zeta(s)} \sum_{j=0}^\infty \frac{1}{f(s)^j} \left(-\frac{\zeta'}{\zeta}(s)\right)^j
\end{align*}
where we have expanded the last term on the first line as a geometric series. This expansion will be valid so long as $\Re(s)$ is large enough so that $\left| \frac{1}{f(s)}  \frac{\zeta'}{\zeta}(s)  \right| \leq 1$. That is, we can take $\Re (s)>1$ and $t \geq 100$.

Since $\Re(s)>1$, we may substitute the following Dirichlet series into the first term on the right hand side of \eqref{eq:logderivZ1}
\begin{align*}
    &\zeta''(s) = \sum_{n=1}^\infty \frac{\log^2(n)}{n^s} &\zeta'(s) &= -\sum_{n=1}^\infty \frac{\log(n)}{n^s}\\
    &\frac{1}{\zeta(s)} = \sum_{n=1}^\infty \frac{\mu(n)}{n^s}
    &-\frac{\zeta'}{\zeta}(s) &= \sum_{n=1}^\infty \frac{\Lambda(n)}{n^s}
\end{align*}
which gives
\[
\frac{\zeta''(s) +  f(s) \zeta'(s)}{\zeta'(s) + f(s) \zeta(s)} = \left(\sum_{n=1}^\infty \frac{\log^2 n}{n^s}  - f(s) \sum_{n=1}^\infty \frac{\log n}{n^s} \right) \frac{1}{f(s)} \sum_{n=1}^\infty \frac{\mu(n)}{n^s}   \sum_{j=0}^\infty \frac{1}{f(s)^j} \left(\sum_{n=1}^\infty \frac{\Lambda(n)}{n^s} \right)^j.
\]

We now employ the two M\"obius identities from Lemma \ref{Mobius} to give
\begin{multline} \label{eq:ConvPlusMobius}
    \left(\sum_{n=1}^\infty \frac{\log^2 n}{n^s}  - f(s) \sum_{n=1}^\infty \frac{\log n}{n^s} \right) \frac{1}{f(s)} \sum_{n=1}^\infty \frac{\mu(n)}{n^s} \\
    = \frac{1}{f(s)} \sum_{n=1}^\infty \frac{\Lambda(n) \log(n) }{n^s} + \frac{1}{f(s)} \sum_{n=1}^\infty \frac{\Lambda_2(n)}{n^s} - \sum_{n=1}^\infty \frac{\Lambda(n)}{n^s}.
\end{multline}

Finally, consider the Dirichlet series coming from the geometric expansion. They will consist of multiple convolutions, and recall we write $\Lambda_k$ to denote $k-1$ convolutions of $\Lambda$ with itself. This means
\begin{multline*}
\sum_{j=0}^\infty \frac{1}{f(s)^j} \left(\sum_{n=1}^\infty \frac{\Lambda(n)}{n^s} \right)^j = 1 + \frac{1}{f(s)} \sum_{n=1}^\infty \frac{\Lambda(n)}{n^s} + \frac{1}{f(s)^2} \sum_{n=1}^\infty \frac{\Lambda_2(n)}{n^s} \\
+ \frac{1}{f(s)^3} \sum_{n=1}^\infty \frac{\Lambda_3(n)}{n^s}  + \frac{1}{f(s)^4} \sum_{n=1}^\infty \frac{\Lambda_4 (n)}{n^s} + \dots
\end{multline*}

Consider the last two Dirichlet sums in \eqref{eq:ConvPlusMobius} multiplied with the above terms coming from the geometric expansion. Note they telescope nicely. Using the above, we have
\begin{multline*}
\left(\frac{1}{f(s)} \sum_{n=1}^\infty \frac{\Lambda_2(n)}{n^s} - \sum_{n=1}^\infty \frac{\Lambda(n)}{n^s} \right) \left(1 + \frac{1}{f(s)} \sum_{n=1}^\infty \frac{\Lambda(n)}{n^s} + \frac{1}{f(s)^2} \sum_{n=1}^\infty \frac{\Lambda_2(n)}{n^s} + \dots\right) \\
= \frac{1}{f(s)} \sum_{n=1}^\infty \frac{\Lambda_2(n)}{n^s} + \frac{1}{f^2(s)} \sum_{n=1}^\infty \frac{\Lambda_3(n)}{n^s} + \frac{1}{f(s)^3} \sum_{n=1}^\infty \frac{\Lambda_4(n)}{n^s} + \dots \\
- \sum_{n=1}^\infty \frac{\Lambda(n)}{n^s} - \frac{1}{f(s)} \sum_{n=1}^\infty \frac{\Lambda_2(n)}{n^s} - \frac{1}{f^2(s)} \sum_{n=1}^\infty \frac{\Lambda_3(n)}{n^s} - \dots
= -\sum_{n=1}^\infty \frac{\Lambda(n)}{n^s}
\end{multline*}

Returning to \eqref{eq:ConvPlusMobius}, we still will have contributions coming from the convolutions of $\Lambda(n)\log(n)$ with $\Lambda_k(n)$, which we will write as $ (\Lambda \cdot \log) \ast \Lambda_k $. For clarity, this means
\[
\left((\Lambda \cdot \log) \ast \Lambda_k \right)(n) = \sum_{d \mid n}  \Lambda(d) \log(d) \Lambda_k (\tfrac{n}{d}).
\]

Putting everything together we have
\begin{multline*}
\frac{Z_1'}{Z_1}(s) = - \sum_{n=1}^\infty \frac{\Lambda(n)}{n^s} + \frac{1}{f(s)} \sum_{n=1}^\infty \frac{\Lambda(n) \log(n) }{n^s} + \frac{1}{f(s)^2} \sum_{n=1}^\infty \frac{\left((\Lambda \cdot \log) \ast \Lambda \right)(n) }{n^s}\\
+ \frac{1}{f(s)^3} \sum_{n=1}^\infty \frac{\left((\Lambda \cdot \log) \ast \Lambda_2 \right)(n) }{n^s}  + \frac{1}{f(s)^4} \sum_{n=1}^\infty \frac{\left((\Lambda \cdot \log) \ast \Lambda_3\right)(n) }{n^s} + \dots + O \left( \frac{1}{t \log t} \right).
\end{multline*}
\end{proof}

We will be applying the method of stationary phase at various points throughout our argument. The following version of this result can be found in \cite{KaraYil11}, which is based on a similar result due to Gonek \cite{Gonek84}, where he took $k$ in the following lemma to be a non-negative integer.
\begin{lemma}\label{lem:StatPhase}
Let $\{b_m\}_{m=1}^{\infty}$ be a sequence of complex numbers such that for any $\varepsilon>0$, $b_m \ll m^\varepsilon$. Let $c>1$ and suppose $|k| = o(\log T)$ as $T \rightarrow \infty$. Then for $T$ sufficiently large,
\[
\frac{1}{2 \pi} \int_{1}^{T} \chi (1-c-it) \left( \log \frac{t}{2 \pi} \right)^k \left( \sum_{m=1}^{\infty} b_m m^{-c -it} \right) \ dt = \sum_{1 \leq m \leq \frac{T}{2 \pi}} b_m (\log m)^k + O(T^{c-\frac{1}{2}} (\log T)^k).
\]
\end{lemma}

\section{Proof of Theorem \ref{thm:LowerOrderTerms}}\label{Proof}
\subsection{Setting up the proof}
The start of our proof follows that of Conrey and Ghosh \cite{ConGho85} but we recreate it here for completeness sake. As we have already noted, under the Riemann Hypothesis the zeros of $Z'(t)$ are interlaced with those of $Z(t)$, and so the set of points where $|\zeta (1/2 +it)|$ achieves its maxima are exactly the points where $Z'(t) = 0$. We know that $Z_1(1/2+it) = 0$ if and only if $Z'(t) =0$ by part (1) of Lemma \ref{lem:CGLemma}.

Let  $\rho_{1} = \beta_{1} + i \gamma_{1}$ be the zeros of $Z_1(s)$. By part (3) of Lemma \ref{lem:CGLemma}  $Z_1(s)$ has $O(\log T)$ zeros off the critical line (that is, $\beta_1 \neq 1/2$). At these zeros, using part (4) of Lemma \ref{lem:CGLemma} and the Lindel\"of bounds on zeta, we have
\[
\sum_{\substack{0< \gamma_{1} \leq T \\ \beta_1 \neq 1/2}} \zeta (\rho_{1}) \zeta(1-\rho_1) \ll T^{1/9+\epsilon} T^\epsilon \log T \ll T^{1/8}.
\]

Therefore if $\gamma$ and $\gamma^+$ denote consecutive zeros of the zeta function (which under the Riemann Hypothesis are real zeros of $Z(t)$) we have
\begin{equation*}
    \sum_{0< \gamma \leq T} \max_{\gamma \leq t \leq \gamma^+} \left| \zeta \left( \frac{1}{2} + it \right) \right|^2 =  \sum_{\substack{0 < \gamma_1 \leq T \\ \beta_1 = 1/2}} \left| \zeta \left( \frac{1}{2} + i\gamma_1 \right) \right|^2 .
\end{equation*}

We can write this as an integral using Cauchy's theorem,
\begin{align*}\label{eq:Cauchy}
\frac{1}{2\pi i} \int_{\mathcal{C}} \frac{Z_1'}{Z_1}(s) \zeta (s) \zeta (1-s) \ ds &= \sum_{0 < \gamma_1 \leq T} \zeta \left( \rho_1 \right) \zeta \left( 1-\rho_1 \right) \\
&= \sum_{\substack{0 < \gamma_1 \leq T \\ \beta_1 = 1/2}} \left| \zeta \left( \frac{1}{2} + i\gamma_1 \right) \right|^2 + O\left(T^{1/8} \right)
\end{align*}
where $\mathcal{C}$ is a positively oriented contour with vertices $c + i$, $c + iT$, $1-c+iT$, and $1-c +i$, where $c=1+1/\log T$. We may assume, without loss of generality, that the distance from the contour to any zero $\rho_1$ of $Z_1(s)$ is uniformly $\gg 1/\log T$.

\begin{remark}
    Note that this contour $\mathcal{C}$ is the first time we diverge from the method of Conrey and Ghosh \cite{ConGho85}, since their $c$ is such that the contour lies inside the critical strip, and they go from $T$ to $T+U$, where $U=T^{3/4}$. They do this because in their definition of $Z_1(s)$ they use the approximation
    \[
    - \frac12 \frac{\chi'}{\chi}(s) \approx \frac12 \log \frac{T}{2 \pi}
    \]
    to obtain
    \[
    Z_1(s) \approx \zeta'(s) + \frac12 \log \frac{T}{2 \pi} \zeta (s).
    \]
    This approximation for $\chi'/\chi (s)$ relies on $T$ staying roughly the same size. They then use dyadic sums to get their result for length $T$. Since we keep the factor of $\chi'/\chi (s)$ until late in our proof, we are able to take the height of our contour of length $T$ and avoid any potential issues with dyadic summation.
\end{remark}

\subsection{Initial manipulations of the Cauchy integral}

By considering the Hadamard product for $Z_1(s)$ (which has $N(T)$ zeros on the critical line and $O(\log T)$ zeros off the critical line), one can show that
\[
\frac{Z_1'}{Z_1}(s) = \sum_{|s-\rho_1|<1} \frac{1}{s-\rho} + O(\log T)
\]
so if $|T-\rho_1| \gg 1/\log T$ for all zeros $\rho_1$ of $Z_1$, we have
\[
\frac{Z_1'}{Z_1}(\sigma + i T) \ll (\log T)^2
\]
uniformly for $-1<\sigma \leq 2$.

Therefore, using the Lindel\"of bounds, the integral along the horizontal sides of the contour are bounded by $O \left( T^{c-1/2+\epsilon} \right)$.

All that remains to calculate are the two vertical sides. That is,
\begin{multline*}
    \frac{1}{2\pi i} \int_{\mathcal{C}} \frac{Z_1'}{Z_1}(s) \zeta (s) \zeta (1-s) \ ds \\
    = \frac{1}{2\pi i} \int_{c+i}^{c+iT} \frac{Z_1'}{Z_1}(s) \chi(1-s) \zeta^2 (s) \ ds + \frac{1}{2\pi i} \int_{1-c+i}^{1-c+iT} \frac{Z_1'}{Z_1}(s) \chi(s) \zeta^2 (1-s) \ ds + O \left( T^{c-1/2+\epsilon} \right)\\
    = I_1 + I_2 + O \left( T^{c-1/2+\epsilon} \right),
\end{multline*}
where we have used the functional equation for $\zeta (s)$ in each integral.

Next, we use \eqref{eq:logderivZ1FE}, that is, the logarithmic derivative of the functional equation for $Z_1(s)$, together with using the change of variables $s \mapsto 1-s$ to write
\[
I_2 = \overline{I_1} - \frac{1}{2\pi i} \int_{1-c+it}^{1-c+iT} \frac{\chi'}{\chi}(s) \zeta (s) \zeta (1-s) \ ds.
\]
Therefore we have
\begin{equation*}
    \frac{1}{2\pi i} \int_{\mathcal{C}} \frac{Z_1'}{Z_1}(s) \zeta (s) \zeta (1-s) \ ds
    = 2\Re(I_1) - \frac{1}{2\pi i} \int_{1-c+it}^{1-c+iT} \frac{\chi'}{\chi}(s) \zeta (s) \zeta (1-s) \ ds +O \left( T^{c-1/2+\epsilon} \right).
\end{equation*}
We may now move the line of integration in the twisted integral to the half-line, obtaining
\[
    - \frac{1}{2 \pi} \int_{1}^{T}\frac{\chi^{\prime}}{\chi}\left(\frac{1}{2}+it\right)\left|\zeta \left(\frac{1}{2}+it \right)\right|^2 \ dt + O \left( T^{c-1/2+\epsilon} \right).
\]
This is exactly the second twisted moment integral evaluated in Lemma \ref{lem:TwistSecond}.

\begin{remark}
    Here is a trivial point of deviation from Conrey and Ghosh \cite{ConGho85}. Since they are only after leading order behaviour, they ignore the lower-order terms from the twisted second moment, whereas since we are after a full asymptotic, we are obviously required to keep track of all lower-order terms throughout the argument.
\end{remark}

Therefore, to complete the argument,  we need to calculate $2 \Re (I_1)$.

\subsection{Calculating the main terms}
We begin this subsection by giving a rough overview of what we are aiming to do here as it is the most complicated part of the argument. Since we are to the right of the abscissa of absolute convergence for our $L$-functions, we can write them as their respective Dirichlet series. After this, we use the method of stationary phase to write each integral as a sum of length $T$. Finally, we use Perron's formula to rewrite these sums as an integral, and after shifting the integral past the pole at $s=1$, we may collect residues to obtain a full asymptotic.

We already know that we can write
\begin{equation}\label{eq:zetasquared}
    \zeta^2 (s) = \sum_{n=1}^{\infty} \frac{d(n)}{n^s}
\end{equation}
for $\Re (s) >1$, and where $d(n) = \sum_{d|n} 1$.

We have shown that in Lemma \ref{lem:Dirichletlogderiv} for $\Re(s)>1$ and $t \geq 100$ that
\[
\frac{Z_1'}{Z_1}(s) = \sum_{n=1}^\infty \frac{a(n,s)}{n^s}
\]
where
\begin{equation*}
    a(n,s) = -\Lambda (n) + \sum_{k=1}^{\infty} \frac{1}{f(s)^k} a_k(n)
\end{equation*}
where $a_k(n)$ is given by \eqref{eq:ak(n)} and $f(s)$ is given by \eqref{eq:f(s)}.

While the integral $I_1$ has been from $c+i$ to $c+iT$, we now take it from $c+100i$ to $c+iT$ at the cost of an error of size $O(1)$.

By absolute convergence, we can swap the integral and the sum over $k$, obtaining
\begin{align*}
        I_1 &= \frac{1}{2 \pi i} \int_{c + 100i}^{c+ iT} \chi (1-s) \sum_{n=1}^{\infty} \frac{-\Lambda (n)}{n^s} \zeta (s) ^2 \ ds  \\
        &\qquad+ \sum_{k=1}^{\infty} \frac{1}{2 \pi i} \int_{c+100i}^{c + iT} \chi (1-s) \frac{1}{f(s)^k} \sum_{n=1}^{\infty} \frac{a_k(n)}{n^s} \zeta (s) ^2 \ ds \\
        &= J_{0} + \sum_{k=1}^\infty J_{k},
\end{align*}
say.

\begin{lemma}
    Under the Riemann Hypothesis we have
    \[
    2J_0 - \frac{1}{2\pi}\int_{1}^{T}\frac{\chi^{\prime}}{\chi} \left(\frac{1}{2} + it \right) \left|\zeta \left(\frac{1}{2} + it \right) \right|^2 \ dt = O\left(T^{1/2+\epsilon}\right) .
    \]
\end{lemma}

\begin{proof}
Extending $J_0$ back to be over $1$ to $T$, and remembering that $c=1+1/\log T$ we may use a slight variation of Gonek's Lemma 5 \cite{Gonek84} to write the integral $J_0$ as
\begin{equation*}
J_0 = - \sum_{nm \leq T/2\pi} \Lambda (n) d(m) + O\left(T^{1/2+\epsilon} \right) .
\end{equation*}
By Perron's formula we have
\begin{equation*}
   - \sum_{nm \leq x} \Lambda (n) d(m) = \frac{1}{2 \pi i} \int_{3/2-i T}^{3/2+i T} \frac{\zeta'}{\zeta}(s) \zeta (s)^2 \frac{x^s}{s} \ ds + O\left(T^{1/2+\epsilon}\right)
\end{equation*}
where $x=T/2 \pi$.

By the standard methods of shifting the contour to the line with $\Re(s)=1/2+1/\log T$, and evaluating the residue from the pole at $s=1$, we have
\begin{equation}\label{eq:J0}
    J_0 = -\frac{T}{4 \pi} \left( \log \frac{T}{2 \pi} \right)^2 + \frac{T}{2 \pi} \log \frac{T}{2 \pi} (1-\gamma_0) + \frac{T}{2 \pi} (-1+\gamma_0) + O(T^{1/2 + \varepsilon})
\end{equation}
where $\gamma_0$ is Euler's constant.

By Lemma \ref{lem:TwistSecond},
\begin{multline*}
\frac{1}{2\pi}\int_{1}^{T}\frac{\chi^{\prime}}{\chi} \left(\frac{1}{2} + it \right) \left|\zeta \left(\frac{1}{2} + it \right) \right|^2 \ dt \\
=     \frac{T}{2 \pi} \left( \log \frac{T}{2 \pi} \right)^2 + \frac{T}{2 \pi} \log \frac{T}{2 \pi} (-2 + 2\gamma_0) + \frac{T}{2 \pi} (2-2\gamma_0) + O\left( T^{1/2 + \varepsilon}\right).
\end{multline*}
This perfectly matches double the evaluation of $J_0$ given in \eqref{eq:J0}, thus proving the lemma.
\end{proof}

That is, we have shown that
\begin{equation}\label{eq:ZetaSumInTermsOfJk}
\sum_{0< \gamma \leq T} \max_{\gamma \leq t \leq \gamma^+} \left| \zeta \left( \frac{1}{2} + it \right) \right|^2 = 2\Re \sum_{k=1}^\infty J_k + O\left( T^{1/2 + \varepsilon}\right)
\end{equation}
where
\begin{equation}\label{eq:Jk}
    J_k = \frac{1}{2 \pi i} \int_{c+100i}^{c + iT}  \chi (1-s)  \frac{1}{f(s)^k} \sum_{n=1}^{\infty} \frac{a_k(n)}{n^s} \zeta (s) ^2 \ ds
\end{equation}
where the $a_k(n)$ are given in \eqref{eq:ak(n)} and $f(s)$ is given in \eqref{eq:f(s)}.

\begin{lemma}\label{lem:JkStationary}
    We can write
    \begin{equation*}
    J_k = 2^k \sum_{nm \leq T/2 \pi}  \frac{a_k(n)}{(\log nm)^k} d(m) + O\left( T^{1/2+\varepsilon} \right).
\end{equation*}
where $a_k(n)$ are given in \eqref{eq:ak(n)} and $d(m)$ is the Divisor function.
\end{lemma}

\begin{proof}
First note that extending the integral to go from $c+i$ rather than $c+100i$ only introduces an error of size $O(1)$ which we will ignore as it will be consumed by other larger error terms. We begin by rewriting \eqref{eq:Jk} with $\zeta (s)^2$ written as its Dirichlet series expansion. That is,
\[
J_k = \frac{1}{2 \pi i} \int_{c+i}^{c + iT}  \chi (1-s)  \frac{1}{f(s)^k} \sum_{n=1}^{\infty} \frac{a_k(n)}{n^s} \sum_{m=1}^{\infty} \frac{d(m)}{m^s} \ ds,
\]
where $d(m)$ is the divisor counting function, with $d(m) = \sum_{d|m} 1$. Also, by \eqref{eq:ChiPrimeOverChi} and \eqref{eq:f(s)}, we can write
\begin{equation*}
f(s) = -\frac{1}{2} \frac{\chi'}{\chi} (s) = \frac{1}{2} \log \frac{t}{2 \pi} + O\left( \frac{1}{t} \right)
\end{equation*}
and so
\begin{equation*}
    \frac{1}{f(s)^k} = \frac{2^k}{(\log t/2\pi)^k} \left( 1 + O\left( \frac{1}{t \log t} \right) \right)^{-k} = \frac{2^k}{(\log t/2\pi)^k} \left( 1 + O\left( \frac{k}{t \log t} \right) \right).
\end{equation*}
Then
\begin{align*}
J_k &= \frac{1}{2 \pi i} \int_{c+i}^{c + iT}  \chi (1-s)  \frac{2^k}{(\log t/2\pi)^k} \left( 1 + O\left( \frac{k}{t \log t} \right) \right) \sum_{n=1}^{\infty} \frac{a_k(n)}{n^s} \sum_{m=1}^{\infty} \frac{d(m)}{m^s} \ ds \\
&= \frac{1}{2 \pi i} \int_{c+i}^{c + iT}  \chi (1-s)  \frac{2^k}{(\log t/2\pi)^k} \sum_{n=1}^{\infty} \frac{a_k(n)}{n^s} \sum_{m=1}^{\infty} \frac{d(m)}{m^s} \ ds + O \left( T^{1/2+\varepsilon} \right).
\end{align*}
By the method of stationary phase as given in Lemma \ref{lem:StatPhase}, the result follows.
\end{proof}

The issue is now clear - there is no obvious Dirichlet series for this sum for which we can apply Perron directly. Instead, we will use evaluate the partial sum of $a_k \ast d$ using Peron's formula, and then use partial summation to calculate $J_k$.

\begin{remark}
    Note that since we have kept $f(s)$ exact throughout this argument, we have genuine terms from the integral part of partial summation (which involes $f'(s)$). This contributes some subleading terms to the asymptotic but since these are only subleading, it will be clear why Conrey and Ghosh are able to set it equal to a constant and obtain their leading order behaviour (which we reaffirm in our argument).
\end{remark}

\begin{lemma}\label{lem:Ak}
    Let
    \[
    A_k(x) = \sum_{1 \leq n_1 n_2 ... n_k m \leq x} \log (n_1) \Lambda (n_1) \Lambda (n_2) \dots \Lambda (n_k) d(m) .
    \]
    which is the partial summation of $(a_k\ast d)(n)$ for $n$ up to $x$, where $a_k$ is given by \eqref{eq:ak(n)}.

    Then for large $x$,
    \[
    A_k(x) = x \sum_{j=0}^{k+2} \frac{c_{k,j}}{(k+2-j)!} (\log x)^{k+2-j} + O\left(x^{1/2+\epsilon}\right)
    \]
where the $c_{k,\ell}$ are the Laurent series coefficients around $s=1$ of
\[
\left(\frac{\zeta '}{\zeta} (s) \right)' \left(-\frac{\zeta '}{\zeta} (s) \right)^{k-1} \zeta^2 (s) \frac{1}{s} = \sum_{\ell=0}^\infty c_{k,\ell} (s-1)^{-k-3+\ell} .
\]
\end{lemma}

\begin{remark}
A long but straightforward calculation shows that the first values of $c_{k,\ell}$ are
\begin{align*}
c_{k,0} &= 1 \\
c_{k,1} &= -1-(k-3)\gamma_0\\
c_{k,2} &= 1+(k-3)\gamma_0 + \frac12(k-1)(k-4) \gamma_0^2 + 2(k-3)\gamma_1
\end{align*}
Here $\gamma_0$ is the Euler–Mascheroni constant and $\gamma_1$ is the first Steiltjes constant.
\end{remark}

\begin{proof}
By the method of Perron, we have
\begin{equation*}
    A_k(x) = \frac{1}{2 \pi i} \int_{3/2-i R}^{3/2+i R} \left(\frac{\zeta '}{\zeta} (s) \right)' \left(-\frac{\zeta '}{\zeta} (s) \right)^{k-1} \zeta^2 (s) \frac{x^s}{s} \ ds + O\left(\frac{x^{3/2+\epsilon}}{R}\right).
\end{equation*}

We now calculate $A_k(x)$ by shifting past the pole at $s=1$ to the line $\Re(s)=1/2+\epsilon$, and evaluating the residue there.

Note that the integrand has a pole of order $k+3$ at $s=1$. Therefore, up to a power saving error term, $A_k(x)$ will equal the residue of the integrand around that point, which means
\[
A_k(x) = x \sum_{j=0}^{k+2} b_{k,j} (\log x)^{k+2-j} + O\left(R^{\epsilon} x^{1/2+\epsilon}\right) + O\left(\frac{x^{3/2+\epsilon}}{R}\right)
\]
for some constants $b_{k,j}$. (Note we have chosen to start the sum at the largest power of $\log x$, as this simplifies some later calculations). Taking $R=x$ optimises the error terms to be $O(x^{1/2+\epsilon})$.

To evaluate the residues, we employ the various Laurent expansions about $s=1$ in the various terms in the integrand. We give the ones we need for ease of reference. First, define the Stieltjes gamma constants to come from the series expansion of zeta around $s=1$, namely
\[
\zeta(s) = \frac{1}{s-1} + \gamma_0 - \gamma_1 (s-1) + \frac{1}{2!} \gamma_2 (s-1)^2 + \dots + (-1)^j \frac{1}{j!} \gamma_j (s-1)^j + \dots
\]
Then the expansions we need are
\begin{align*}
    &\left( \frac{\zeta'}{\zeta}(s)\right)' &=\quad& \frac{1}{(s-1)^2}+\left(-2 \gamma_1-\gamma_0^2\right)+ \left(6 \gamma_0  \gamma_1+3 \gamma_2+2 \gamma_0^3\right) (s-1) + \dots \\
   & \left(- \frac{\zeta'}{\zeta}(s)\right)^{k-1} &=\quad& \frac{1}{(s-1)^{k-1}} +\frac{(1-k) \gamma_0}{(s-1)^{k-2}} + \frac{\frac12k(k-1)\gamma_0^2 + 2(k-1)\gamma_1}{(s-1)^{k-3}} + \dots\\
  &  \zeta^2 (s) &=\quad& \frac{1}{(s-1)^2}+\frac{2 \gamma_0 }{s-1}+\left(\gamma_0^2-2 \gamma_1\right) + \dots\\
&    \frac{1}{s}&=\quad& 1 - (s-1) + (s-1)^2 + ...
\end{align*}
Combining these Laurent expansions allows us to write
\[
\left(\frac{\zeta '}{\zeta} (s) \right)' \left(-\frac{\zeta '}{\zeta} (s) \right)^{k-1} \zeta^2 (s) \frac{1}{s} = \sum_{\ell=0}^\infty c_{k,\ell} (s-1)^{-k-3+\ell}
\]
where we have explicitly given the first few values in the remark following the statement of the lemma.

The residue is the coefficient of $1/(s-1)$ which will come from combining these terms with the $x^s$ in the integrand (which is the only term that contains an $x$). Its expansion about $s-1$ is
\[
x^s = x \left( 1 + (s-1) \log x + \frac{(s-1)^2}{2!} (\log x)^2 + ... + \frac{(s-1)^k}{k!} (\log x)^k +...\right)
\]
so the coefficient of $(\log x)^{k+2-j}$ in the $(s-1)^{-1}$ term equals
\[
b_{k,j} = \frac{c_{k,j}}{(k+2-j)!}
\]
and such a combination is possible for $j$ between $0$ and $k+2$.
\end{proof}

Recall by Lemma~\ref{lem:JkStationary}
\begin{align*}
    J_k &= 2^k \sum_{2 \leq nm \leq T/2 \pi}  \frac{a_k(n)}{(\log nm)^k} d(m) +O\left(T^{1/2+\varepsilon}\right) \\
    &= 2^k \sum_{2 \leq n\leq T/2\pi} \frac{(a_k \ast d)(n)}{(\log n)^k} +O\left(T^{1/2+\varepsilon}\right)
\end{align*}
so we can use partial summation together with with Lemma \ref{lem:Ak} to evaluate $J_k$.

\begin{lemma}\label{lem:JkAsympExpansion}
    Writing $L = \log \frac{T}{2\pi}$, for any $N \geq 1$ we have
    \begin{equation*}
    J_k = \frac{T}{2\pi}  \left( \frac{2^k}{(k+2)!} L^2 + \beta_{k,-1} L +
     \beta_{k,0}  + \sum_{n=1}^N \beta_{k,n} L^{-n} \right)
   + O\left(\frac{T}{L^{N+1}}\right)
    \end{equation*}
where
\[
\beta_{k,-1} = \frac{2^k}{(k+2)!}\left( -2-(k+2)(k-3) \gamma_0 \right),
\]
and
\[
\beta_{k,0} = \frac{2^k \left( (k+2)(k+1) c_{k,2} - k c_{k,0}  + k (k+2) c_{k,1}  \right) }{(k+2)!}
\]
and for $1 \leq n\leq k$
\begin{equation*}
\beta_{k,n} = \frac{2^k c_{k,n+2}}{(k-n)!} + \frac{2^k (n-1)!}{(k-1)!} \sum_{j=0}^{n-1}  \binom{k}{j} c_{k,j+2}
\end{equation*}
and for $n \geq k+1$,
\begin{equation*}
\beta_{k,n} = \frac{2^k (n-1)!}{(k-1)!} \sum_{j=0}^{k} \binom{k}{j} c_{k,j+2}
\end{equation*}
where the $c_{k,\ell}$ are given in Lemma~\ref{lem:Ak}
\end{lemma}

\begin{proof}
Set $f(x) = 1/(\log x)^k$, so $f'(x) = -k/(x (\log x)^{k+1})$, and by partial summation we have
\begin{align*}
    J_k &= 2^k A_k\left(\frac{T}{2\pi}\right) f\left(\frac{T}{2\pi}\right) - 2^k \int_2^{T/2\pi} A_k(x) f'(x) \ dx + O\left(T^{1/2+\epsilon}\right) \\
    &= 2^k \frac{T}{2\pi}\sum_{j=0}^{k+2} \frac{c_{k,j}}{(k+2-j)!} L^{2-j} + 2^k k \sum_{j=0}^{k+2} \frac{c_{k,j}}{(k+2-j)!} \int_2^{T/2\pi}  (\log x)^{1-j} dx + O\left(T^{1/2+\epsilon}\right)
\end{align*}
where $L = \log \frac{T}{2\pi}$, and where
\[
A_k(x) = \sum_{n\leq x} (a_k \ast d)(n) ,
\]
which was evaluated in Lemma~\ref{lem:Ak}.

Note the dominant term (that is, $L^2$) comes from the first sum only, whereas subleading and later terms also have contributions from the second sum. For $j=0$ and $j=1$ in the second sum, one can perform the integrals exactly, obtaining
\begin{multline}\label{eq:JkSumCk}
J_k = \frac{T}{2\pi} \left(\frac{2^k c_{k,0}}{(k+2)!} L^2 + \left[ \frac{2^k c_{k,1}}{(k+1)!} +\frac{k 2^k c_{k,0}}{(k+2)!} \right] L + \left[ \frac{2^k c_{k,2}}{k!} - \frac{k 2^k c_{k,0}}{(k+2)!} + \frac{k 2^k c_{k,1}}{(k+1)!} \right] \right) \\
+  \frac{T}{2\pi} \sum_{j=1}^{k} \frac{2^k c_{k,j+2}}{(k-j)!} \frac{1}{L^{j}} + \sum_{j=1}^{k+1} \frac{k 2^k c_{k,j+1}}{(k+1-j)!} \int_2^{T/2\pi} \frac{1}{(\log x)^{j}} dx +O\left(T^{1/2+\epsilon}\right)
\end{multline}
where we have relabeled the sums on the second line to start from $j=1$ now. Note that the remaining integrals are all negative powers of $\log x$, so evaluate to an infinite chain of decreasing powers of $L$, for example for $1 \leq j \leq N$, we have
\begin{align*}
    \int_{2}^{T/2\pi} \frac{1}{(\log x)^j} \ dx & = \frac{T}{2\pi L^j} + j \int_{2}^{T/2\pi} \frac{1}{(\log x)^{j+1}} \ dx + O(1)\\
    &=  \frac{T}{2\pi} \sum_{\ell=j}^{N} \frac{(\ell-1)!}{(j-1)!}\frac{1}{L^{\ell}}
   + O\left( \frac{T}{L^{N+1}} \right).
\end{align*}

Finally, $\beta_{k,n}$ collects all the coefficients of $L^{-n}$. For example
\begin{align*}
    \beta_{k,-1} &= \frac{2^k c_{k,1}}{(k+1)!} +\frac{k 2^k c_{k,0}}{(k+2)!}  \\
    &= - \frac{2^k}{(k+2)!} \left(2+ (k+2)(k-3)\gamma_0)  \right)
\end{align*}
where we use the remark following the statement of Lemma~\ref{lem:Ak}, to evaluate $c_{k,0}$ and $c_{k,1}$.

If $n\leq k$ then there will be a contribution from the second line of \eqref{eq:JkSumCk} coming from the first sum (with $j=n$) and from all the integrals in the second sum with $j$ between $1$ and $n$. We have (for $n\leq k$)
\begin{align*}
\beta_{k,n} &= \frac{2^k c_{k,n+2}}{(k-n)!}  + \sum_{j=1}^{n} \frac{k 2^k c_{k,j+1}}{(k+1-j)!} \frac{(n-1)!}{(j-1)!} \\
&= \frac{2^k c_{k,n+2}}{(k-n)!}  + \sum_{j=1}^{n} \binom{k}{j-1} \frac{2^k (n-1)! c_{k,j+1}}{(k-1)!}
\end{align*}
whereas if $n \geq k+1$ then there will only be contributions coming from the second sum of the second line of \eqref{eq:JkSumCk}, but every term in that sum contributes. We find
\begin{align*}
\beta_{k,n} &= \sum_{j=1}^{k+1} \frac{k 2^k c_{k,j+1}}{(k+1-j)!} \frac{(n-1)!}{(j-1)!} \\
&=\sum_{j=1}^{k+1} \binom{k}{j-1} \frac{2^k (n-1)! c_{k,j+1}}{(k-1)!}
\end{align*}
These both give the same as in the statement of the lemma, after a trivial relabelling of the index $j$.
\end{proof}

Now we have evaluated $J_k$, in order to complete the proof of the Theorem, we can see from \eqref{eq:ZetaSumInTermsOfJk} that all we must do is sum the contribution from $2 J_k$ over all $k \geq 1$.

We have
\begin{align}\label{eq:asympWithError}
\sum_{0< \gamma \leq T} \max_{\gamma \leq t \leq \gamma^+} \left| \zeta \left( \frac{1}{2} + it \right) \right|^2 &= 2\Re \sum_{k=1}^\infty J_k + O\left( T^{1/2 + \varepsilon}\right) \\
&= \frac{T}{2\pi} \left(\alpha_{-2} L^2 + \alpha_{-1} L + \alpha_0 + \sum_{n=1}^N \frac{\alpha_n }{L^N} \right)+ O_N\left(\frac{T}{L^{N+1}} \right) \notag
\end{align}
where
\[
    \alpha_{-2} = \sum_{k=1}^\infty \frac{2^{k+1}}{(k+2)!} = \frac{e^2-5}{2}
\]
\begin{align*}
    \alpha_{-1} &= \sum_{k=1}^\infty \frac{2^k}{(k+2)!}\left( -2-(k+2)(k-3) \gamma_0 \right)  \\
    &= 5-e^2-10 \gamma_0 +2 e^2 \gamma_0
\end{align*}
and
\begin{align*}
\alpha_{0} &= \sum_{k=1}^\infty \frac{2^{k+1} \left( (k+2)(k+1) c_{k,2} - k c_{k,0}  + k (k+2) c_{k,1}  \right) }{(k+2)!} \\
&=  12 \gamma_1-4 e^2 \gamma_1-5+e^2+10 \gamma_0 -2 e^2 \gamma_0 -4\gamma_0 ^2
\end{align*}
after substituting the values for $c_{k,\ell}$ found in the remark following the statement of Lemma~\ref{lem:Ak}. Note that the $\alpha_{-2}$ coefficient recovers the result of Conrey and Ghosh \eqref{CGLead}. Also,
\[
\alpha_n = \sum_{k=1}^\infty 2\beta_{k,n}
\]
with the $\beta_{k,n}$ given in the previous lemma. This completes the proof of the Theorem.

\section{Graphical Data}\label{Graphs}

Using a simple bisection method, for the first million zeros using Mathematica we found the location of the local maximum of $Z(t)^2$, accurate to to within $10^{-4}$ times the gap of the adjacent zeros. We also found the value of $Z(t_n)^2$ at that point.

Figure~\ref{fig:truth} plots the cumulative total of $Z(t_n)^2$ in blue over the first million values of $t_n$. This takes us to height $T=600,269.97$. In green is the leading order asymptotic of this, $\frac{e^2-5}{4\pi} T \left(\log\frac{T}{2\pi} \right)^2$ first given by Conrey and Ghosh.

    \begin{figure}[ht]
        \centering
        \includegraphics{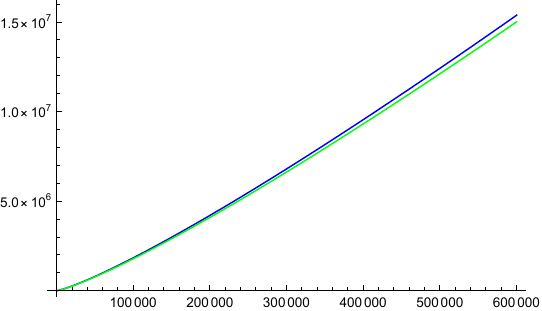}
        \caption{The blue curve shows the true value of $\sum_{0< \gamma \leq T} \max_{\gamma \leq t \leq \gamma^+} \left| \zeta \left( \frac{1}{2} + it \right) \right|^2$. The green curve shows the leading order asymptotic $ \frac{e^2 - 5}{4 \pi} T (\log \frac{T}{2\pi})^2$. The range of the graph is over the first million zeros of zeta.}\label{fig:truth}
    \end{figure}

\begin{figure}
\begin{subfigure}{.5\textwidth}
  \centering
  \includegraphics[width=.8\linewidth]{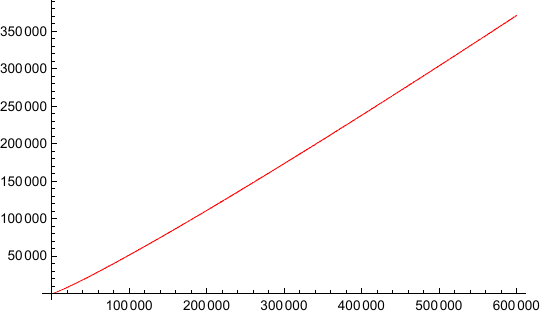}
  \caption{Leading order}
  \label{fig:sfig1}
\end{subfigure}%
\begin{subfigure}{.5\textwidth}
  \centering
  \includegraphics[width=.8\linewidth]{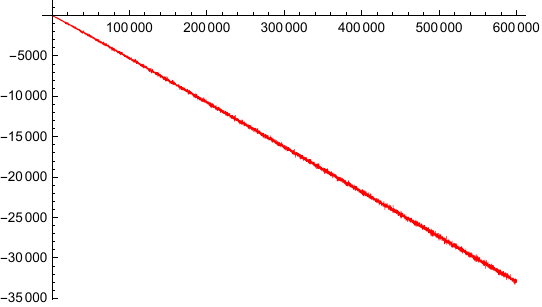}
  \caption{Subleading order}
  \label{fig:sfig2}
\end{subfigure}\\
\begin{subfigure}{.5\textwidth}
  \centering
  \includegraphics[width=.8\linewidth]{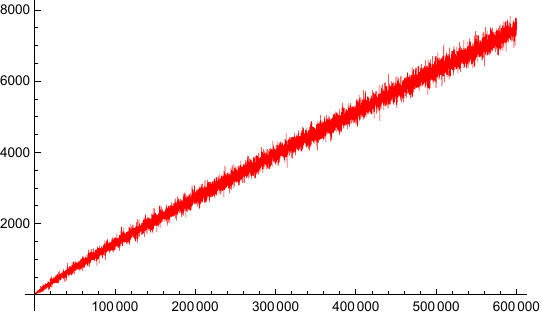}
  \caption{$N=0$}
  \label{fig:sfig3}
\end{subfigure}%
\begin{subfigure}{.5\textwidth}
  \centering
  \includegraphics[width=.8\linewidth]{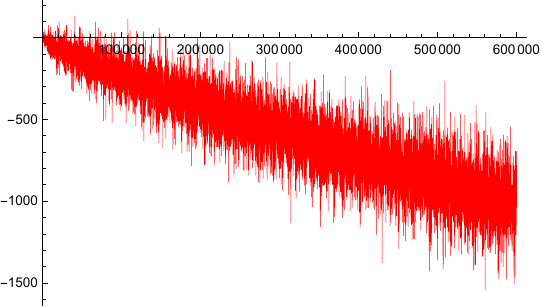}
  \caption{$N=1$}
  \label{fig:sfig4}
\end{subfigure}\\
\begin{subfigure}{.5\textwidth}
  \centering
  \includegraphics[width=.8\linewidth]{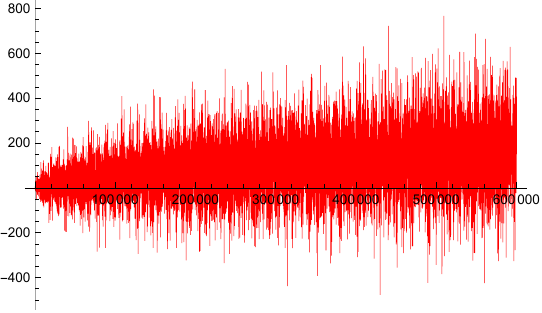}
  \caption{$N=2$}
  \label{fig:sfig5}
\end{subfigure}%
\begin{subfigure}{.5\textwidth}
  \centering
  \includegraphics[width=.8\linewidth]{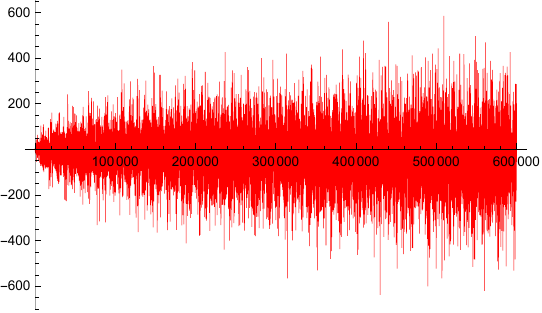}
  \caption{$N=3$}
  \label{fig:sfig6}
\end{subfigure}\\
\begin{subfigure}{.5\textwidth}
  \centering
  \includegraphics[width=.8\linewidth]{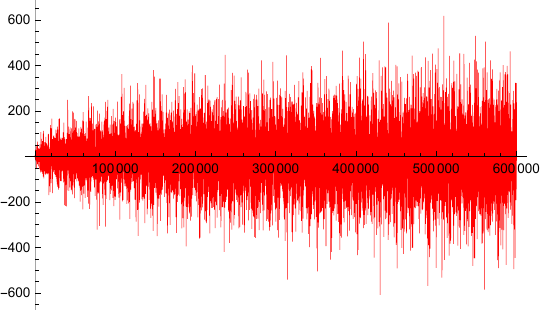}
  \caption{$N=4$}
  \label{fig:sfig7}
\end{subfigure}%
\begin{subfigure}{.5\textwidth}
  \centering
  \includegraphics[width=.8\linewidth]{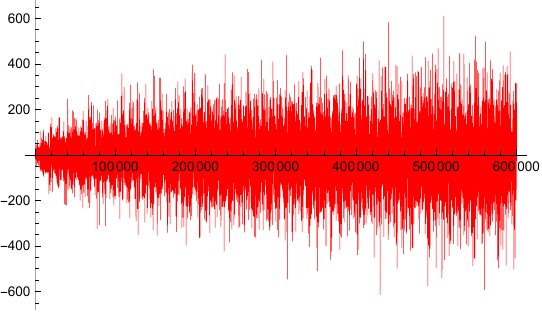}
  \caption{$N=5$}
  \label{fig:sfig8}
\end{subfigure}\\
\caption{The error between the true value and the asymptotic value
given in Theorem~\ref{thm:LowerOrderTerms} for various $N$, plotted over the first million zeros of zeta}
\label{fig:error}
\end{figure}

Figure~\ref{fig:error} shows the error between the true value of
\[
\sum_{0< \gamma \leq T} \max_{\gamma \leq t \leq \gamma^+} \left| \zeta \left( \frac{1}{2} + it \right) \right|^2
\]
and the asymptotic form given in Theorem~\ref{thm:LowerOrderTerms}, in the cases when $-2 \leq N \leq 5$, that is
\[
\frac{e^2-5}{2} \frac{t}{2\pi} \left(\log \frac{t}{2\pi}\right)^2 + \frac{t}{2 \pi} \sum_{n=-1}^{N} \alpha_n \left(\log \frac{t}{2\pi}\right)^{-n}
\]
with the values of $\alpha_n$ for $n=-1,0,1,2,3,4,5$ given in the Theorem. The last three graphs show that $N=3,4,5$ all yield approximately the same size of error term.

We can see the effect of increasing $N$ on the goodness of the asymptotic expansion. Fixing $T$ at the height of the millionth local maximum (that is, $t \approx 600269.96$, the true value of the cumulative sum is $1.53778\times 10^7$, and Table~\ref{tab:ErrorForVariousN} shows the absolute error for varying $N$ at that point.
\begin{center}
\begin{table}
\begin{tabular}{|l|r@{.}l|}
\hline
\multicolumn{1}{|c|}{$N$} & \multicolumn{2}{|c|}{Error} \\
\hline
$-2$ (the leading order) & $371166$&05\\
$-1$ & $-33026$&28\\
0 & $7412$&69\\
1 & $-1072$&47\\
2 & $113$&86\\
3 & $-91$&45\\
4 & $-54$&63\\
5 & $-62$&32\\
6 & $-60$&88\\
7 & $-61$&27\\
8 & $-61$&23\\
9 & $-61$&27\\
10 & $-61$&29\\
    \hline
    \end{tabular}

    \caption{Table showing the error in the asymptotic approximation for various $N$ at the millionth local maximum point}
    \label{tab:ErrorForVariousN}
    \end{table}
\end{center}

Theorem~\ref{thm:LowerOrderTerms} is stated for a fixed stopping point $N$. It is well known that a typical asymptotic expansion does not converge as an infinite series (hence why they should be treated as finite series), but one can nonetheless allow the truncation to grow with $T$. In this case we have carried through an additional error of size $O\left(T^{1/2+\epsilon}\right)$. Our final calculation is to perform a very rough estimate of the optimal truncation, that is roughly where one expects the error from the asymptotic expansion to equal $O\left(T^{1/2+\epsilon}\right)$.

In Theorem~\ref{thm:LowerOrderTerms} we found the coefficient of $(\log\frac{T}{2\pi})^{-n}$ equals
\[
\alpha_n = \sum_{k=n}^\infty \frac{2^{k+1}  c_{k,n+2} }{(k-n)!} +  (n-1)! \sum_{k=1}^\infty \frac{2^{k+1}}{(k-1)!} \sum_{j=0}^{\min\{k,n-1\}}  \binom{k}{j} c_{k,j+2}.
\]

We will show that these coefficients grow when $n$ gets large, and indeed grow faster than $(\log\frac{T}{2\pi})^{n}$, so their ratio eventually grows with $n$. Some numerical values of these coefficients are given in Table~\ref{tab:alphan}.

\begin{center}
\begin{table}
\begin{tabular}{|c|r@{.}l||c|r@{.}l||c|r@{.}l|}
\hline
$n$ & \multicolumn{2}{|c||}{$\alpha_n$} & $n$ & \multicolumn{2}{|c||}{$\alpha_n$} & $n$ & \multicolumn{2}{|c|}{$\alpha_n$} \\
\hline
1 & $1$&$02$ & 7 & $106$&$90$ & 13 & $1$&$04\times 10^7$\\
2 & $-1$&$63$ & 8 & $-124$&$52$ & 14 & $1$&$35\times 10^8$\\
3 & $3$&$24$ & 9 & $1485$&$29$ & 15 & $1$&$89\times 10^9$\\
4 & $-6$&$66$ & 10 & $6209$&$69$ & 16 & $2$&$83\times 10^{10}$\\
5 & $15$&$95$ & 11 & $83003$&$56$ & 17 & $4$&$53\times 10^{11}$\\
6 & $-34$&$23$ & 12 & $851308$&$97$ & 18 & $7$&$70\times 10^{12}$\\
    \hline
    \end{tabular}
    \caption{Table showing the numerical size of the values of $\alpha_n$}
    \label{tab:alphan}
    \end{table}
\end{center}

Crudely estimating the size of $c_{k,\ell}$ for large $k$ and $\ell$, we find
\begin{align*}
\left|\sum_{k=1}^\infty \frac{2^{k+1}}{(k-1)!} \sum_{j=0}^{\min\{k,n-1\}}  \binom{k}{j} c_{k,j+2}\right| &\ll \sum_{k=1}^\infty \frac{2^{k+1}}{(k-1)!} \sum_{j=0}^{k}  \binom{k}{j} \left|c_{k,j+2}\right| \\
&\ll 1
\end{align*}
as $n\to\infty$ and the rate of growth of
\[
\sum_{k=n}^\infty \frac{2^{k+1}  c_{k,n+2} }{(k-n)!}
\]
is bounded by $C^n$ for some $C>1$, so considerably smaller than $(n-1)!$.

Therefore $|\alpha_n| \ll (n-1)!$ for large $n$, a result entirely consistent with asymptotic series looking like ``factorial over power''.

So for large $T$, one expects to see $\alpha_n (\log\frac{t}{2\pi})^{-n}$ initially decrease before increasing without bound. Normally, one finds an ``optimal truncation point'' but for our purposes equation \eqref{eq:asympWithError} shows we simply need to find the smallest $N$ such that 
\[
\frac{\alpha_{N}}{\left(\log \frac{T}{2\pi} \right)^N} \ll T^{-1/2 + \epsilon}.
\]
Replacing $\alpha_N$ with $N!$ and taking logarithms and using the crudest form of Stirling's approximation, this will be the smallest $N$ such that
\[
N \log N -N - N \log\log T \ll (-1/2+\epsilon) \log T
\]
for large $T$, and this has solution
\[
N = \eta  \log T
\]
with $\eta$ satisfying $\eta \log \eta - \eta = -1/2$ which has a solution $\eta = 0.186\dots$ (with a larger, irrelevant, solution at $\eta = 2.155\dots$).

That is, if we set $N \approx 0.2 \log T$, then we expect
\[
   \sum_{0< \gamma \leq T} \max_{\gamma \leq t \leq \gamma^+} \left| \zeta \left( \frac{1}{2} + it \right) \right|^2   =  \frac{e^2-5}{2} \frac{T}{2\pi} L^2 + \frac{T}{2 \pi} \sum_{n=-1}^{N} \frac{\alpha_n}{L^n} + O\left(T^{1/2+\epsilon}\right).
\]

\bibliographystyle{abbrv}

\bibliography{bibliography}	

\end{document}